\documentclass[a4paper,12pt,reqno]{amsart}
\usepackage{comment}
\usepackage[utf8]{inputenc}
\usepackage{latexsym}
\usepackage{amssymb}
\usepackage{amsmath}
\usepackage{amsfonts}
\usepackage[T1]{fontenc}
\usepackage{enumerate}
\usepackage{xcolor}
\usepackage{url}
\usepackage{mathrsfs}
\usepackage{ulem}
\newtheorem{theorem}{Theorem}[section]
\newtheorem{lemma}[theorem]{Lemma}

\newtheorem{corollary}[theorem]{Corollary}
\newtheorem{example}[theorem]{Example}

\newtheorem{problem}[theorem]{Problem}

\def\fm{\mathrm{FM}}

\title[]{
On the Fischer-Musz\'ely equation for the positive cones of $C^*$-algebras
}

\author[D.~Hirota]{
Daisuke Hirota
}
\address{
National Institute of Technology, Tsuruoka College, Yamagata 997-8511, Japan
}
\email{dhirota@tsuruoka-nct.ac.jp
}

\author[J.~Oppekepenguin]{
Jyamira Oppekepenguin
}
\thanks{Corresponding author: Jyamira Oppekepenguin,\\ 
Email: 
oppekepnguin@gmail.com
}
\address{
IKAKEN+HIBIKEN, Niigata, Japan
}
\email{oppekepenguin@gmail.com}


\begin{document}

\begin{abstract}
We study the Fischer-Musz\'ely functional equation on the positive semidefinite and positive definite cones of unital $C^*$-algebras. We prove that every bijective map between the positive semidefinite cones satisfying the Fischer-Musz\'ely equation is of the form $T(\cdot)=cJ(\cdot)c$, where $J$ is a Jordan 
$*$-isomorphism and $c$ is a positive invertible element. As a consequence, we obtain an analogous characterization for bijective maps between the positive definite cones of unital $C^*$-algebras.
\end{abstract}

\maketitle

\section{Introduction}\label{s1}
Fischer and Musz\'ely \cite{fm} initiated a study on 
a norm-type functional equation now known as the Fischer-Musz\'ely functional equation.  
A map $T\colon Q\to X$ from a semigroup $(Q,+)$ into a Banach space $X$ is called a $\fm$-map if it satisfies the Fischer-Musz\'ely equation (hereafter, the  $\fm$-equation) 
\[
\|T(a+b)\|=\|T(a)+T(b)\|
\]
for every pair $a,b\in Q$. 
The main problem is to determine whether the map is additive, or more generally, to characterize its form. 
 A detailed and insightful account of the developments
related to this equation is provided by Ger \cite{ger}.

Positive cones constitute one of the fundamental geometric and order-theoretic structures associated with $C^*$-algebras. 
Various geometric and order-theoretic structures on positive cones have been extensively studied, and many important classes of transformations are characterized by preserving such structures including in \cite{molnarbook,hamo,mol17,mol19,chaba20,mol20,dongli,hao,moi}. 
 
According to Hirota  \cite{piro}, the following problem was proposed by  Moln\'ar.
\begin{problem}[Moln\'ar]\label{molnar}
   Suppose that $T$ is a surjective $\fm$-map between the positive semidefinite cones of two $C^*$-algebras. 
   Does it follow that $T$ is additive?
\end{problem}
$\fm$-maps between positive cones (or positive semidefinite cones) of $C^*$-algebras are of particular interest because of its close connections with isometry theory and the characterization of Jordan $*$-isomorphisms on the underlying $C^*$-algebras. 
A major difficulty is that the map is defined only on the positive cone, rather than on the whole $C^*$-algebra. 
Note that a surjective $\fm$-map between entire $C^*$-algebras is additive by a theorem of Tabor \cite[Corollary 1]{t}.
Although $\fm$-maps do not arise explicitly in quantum information theory, the present problem shares the general philosophy of characterizing transformations by geometric or order-theoretic structures rather than by algebraic assumptions.
Affirmative answers to Problem \ref{molnar} were obtained in \cite{piro,shmi} for commutative $C^*$-algebras. The purpose of the present paper is to solve Problem \ref{molnar} for bijections on arbitrary unital C*-algebras. Our proof relies on a reduction to Kadison's theorem \cite[Theorem 2 and Corollary 5]{kadison} and provides a new approach to the characterization of $\fm$-maps on positive cones.

A similar question can also be posed for ordered Banach spaces.
On the other hand, a simple counterexample shows that an $\fm$-map between the positive cones of certain $L^1$ spaces need not be additive (see \cite[Example 4.1]{ztw}).

\section{Main results}\label{s2}
We will start by clarifying the notation and introducing the necessary definitions and properties that we will use throughout the paper. 
In this paper, we use $A$ and $B$ to denote unital $C^*$-algebras. We always write the unit in a $C^*$-algebra by $e$. 
We denote 
\[
A_{SA}=\{a\in A\colon a=a^*\},
\]
the Jordan algebra of all self-adjoint elements in $A$. 
The positive semidefinite cone is
\[
A_+=\{a\in A_{SA}\colon 0\le a\},
\]
and the positive definite cone is
\[
A_{++}=\{a\in A_+\colon \text{$a$ is invertible in $A$}\}.
\]
We denote $0<a$ for an element $a$ in $A_{++}$, for simplicity.

The primary result of this paper is the following.
\begin{theorem}\label{psdc}
    Suppose that $T\colon A_{+}\to B_{+}$ is a bijection.
    Then $T$ satisfies the $\fm$-equation 
    \[
    \|T(a+b)\|=\|T(a)+T(b)\|
    \]
    for every pair $a,b\in A_{+}$ if and only if 
    there is a Jordan $*$-isomorphism $J\colon A\to B$ such that 
    $T(a)=T(e)^\frac12 J(a) T(e)^\frac12$ for every $a\in A_{+}$.
\end{theorem}
Hirota \cite[Theorem 1.1]{piro} proved that the surjective $\fm$-map between commutative $C^*$ algebras without assuming the unit is additive and positive homogeneous. Applying a theorem of Hirota, Shibata and Miura \cite{shmi} give a complete description of the form of a bijective $\fm$-map between commutative $C^*$-algebras. 

As a corollary of Theorem \ref{psdc} we have the following. 
\begin{corollary}\label{mean}
Suppose that $T\colon A_+\to B_+$ is a bijection. 
Then the following are equivalent.
\begin{itemize}
    \item[(1)] $T$ satisfies the Cauchy equation, i.e., $T(a+b)=T(a)+T(b)$ for every pair $a,b\in A_+$,
    \item[(2)] $T$ satisfies the Jensen equation, i.e., 
    $T\left(\frac{a+b}{2}\right)=\frac{T(a)+T(b)}{2}$ for every pair $a,b\in A_+$,
    \item[(3)]$T$ satisfies the $\fm$-equation, i.e., $\|T(a+b)\|=\|T(a)+T(b)\|$ for every pair $a,b\in A_+$, 
    \item[(4)] $T$ satisfies the equation $\left\|T\left(\frac{a+b}{2}\right)\right\|=\left\|\frac{T(a)+T(b)}{2}\right\|$
for every pair $a,b\in A_+$, 
\item[(5)] there is a Jordan $*$-isomorphism $J\colon A\to B$ such that 
    $T(a)=T(e)^\frac12 J(a)T(e)^\frac12$ 
    for every $a\in A_{+}$.
\end{itemize}
\end{corollary}
Moln\'ar \cite[Proposition 1]{mol17} characterized  bijections between the positive definite cones of unital $C^*$-algebras satisfying the Jensen equation. The following corollary provides an anlogous characterization for bijections satisfying the Fischer-Musz\'ely equation.
\begin{corollary}\label{pc}
    Suppose that $T\colon A_{++}\to B_{++}$ is a bijection. Then $T$ satisfies the $\fm$-equation 
    \[
    \|T(a+b)\|=\|T(a)+T(b)\|
    \]
    for every pair $a,b\in A_{++}$ if and only if there is a Jordan $*$-isomorphism $J\colon A\to B$ such that 
    $T(a)=T(e)^\frac12 J(a)T(e)^\frac12$ 
    for every $a\in A_{++}$.
\end{corollary}
\begin{theorem}\label{ontopsdc}
    Suppose that $T\colon A_{+}\to B_{+}$ is a surjection such that $T^{-1}(B_{++})\subset A_{++}$. 
    Then $T$ satisfies the $\fm$-equation 
    \[
    \|T(a+b)\|=\|T(a)+T(b)\|
    \]
    for every pair $a,b\in A_{+}$
    if and only if     
     there is a Jordan $*$-isomorphism $J\colon A\to B$ such that  
    $T(a)=T(e)^\frac12 J(a) T(e)^\frac12$ for every $a\in A_{+}$.
\end{theorem}
In the above theorem, the corresponding algebras $A$ and $B$ are Jordan $*$-isomorphic. Without assuming $T^{-1}(B_{++})\subset A_{++}$, we have a following example. 
\begin{example}
    Let $A_1$ and $A_2$ be unital $C^*$-algebras. Let $T\colon A_{1+}\oplus A_{2+}\to A_{1+}$ (resp. $T\colon A_{1++}\oplus A_{2++}\to A_{1++}$) be defined as $T(a\oplus b)=a$ for every pair $a\in A_{1+}$ (resp. $a\in A_{1++}$) and $b\in A_{2+}$ (resp. $b\in A_{2++}$). Then $T$ is an additive surjection. Hence, it satisfies the $\fm$-equation. 
    In this case, $A_1\oplus A_2$ and $A_1$ need not be Jordan $*$-isomorphic.
\end{example}
\section{Preliminary Lemmata}\label{spre}
Throughout the section, 
$T\colon A_{+}\to B_{+}$ is a {\it surjection} which is an $\fm$-map in the sense that it satisfies the $\fm$-equation $\|T(a+b)\|=\|T(a)+T(b)\|$ for every pair $a,b\in A_{+}$. 
Note that we do not apply the injectivity of $T$ in the proofs of Lemmata \ref{0} through  
\ref{e13,e14,21}.

In this section, $c_0$ denotes any element in $A_+$ such that $T(c_0)=e$. 
\begin{lemma}\label{0}
    We have $T(0)=0$.
\end{lemma}
\begin{proof}
    As 
  \[
  \|T(0)\|=\|T(0+0)\|=\|T(0)+T(0)\|,
  \]
  we have $T(0)=0$. 
\end{proof}
We easily establish that $T$ is order-preserving.
\begin{lemma}\label{p}
    The map $T$ preserves the order; i.e., $a\le b$ ensures that $T(a)\le T(b)$ for every pair $a,b\in A_{+}$.
\end{lemma}
\begin{proof}
    Suppose that $a,b\in A_{+}$ with $a\le b$. Put $c=b-a$. Then $0\le c$ and $a+c=b$. We have by the $\fm$-equation that
    \begin{multline*}
\|T(b)+T(x)\|=\|T(a+c)+T(x)\|\\
=\|T(a+c+x)\|=\|T(a+x)+T(c)\|
    \end{multline*}
    for every $x\in A_+$. 
    Since $0\le T(a+x), T(c)$, we ensure by the FM-equation that
    \[
    \|T(a+x)+T(c)\|\ge \|T(a+x)\|=\|T(a)+T(x)\|.
    \]
    Thus we have
    \begin{equation}\label{3.1}
    \|T(b)+T(x)\|\ge \|T(a)+T(x)\|.
    \end{equation}
    As $T$ is surjective, Lemma 2.6 in \cite{dongli} 
    ensures       
    \[
    T(a)\le T(b).
    \]
\end{proof}
If one could easily prove that $T^{-1}$ is also order-preserving whenever $T$ is a bijective  map, then the remainder of the proof of Theorem \ref{psdc} would not be difficult. Unfortunately, we have been unable to find such a proof. Consequently, the full argument is required to establish Theorem \ref{psdc}.
\begin{lemma}\label{4}
For any positive $\alpha$, we have $T(\alpha c_0)=\alpha e$. 
\end{lemma}
\begin{proof}
Note that there exists $c_0\in A_+$ with $T(c_0)=e$ as $T$ is a surjection. 
    First, we prove by induction on $n$ that 
    \begin{equation}\label{eq3.2-0}
    T\left(2^{-n}c_0\right)=2^{-n}e
    \end{equation}
    for every positive integer $n$. 
Suppose that $n=1$. We have by the $\fm$-equation that 
\begin{multline*}
\|T(2^{-1}c_0)\|=2^{-1}\|T(2^{-1}c_0)+T(2^{-1}c_0)\|
\\
=2^{-1}\|T(2^{-1}c_0+2^{-1}c_0)\|=2^{-1}\|T(c_0)\|=2^{-1}\|e\|=2^{-1}.
\end{multline*}
As the norm coincides with the spectral norm on self-adjoint elements, we observe that
\begin{equation}\label{n1}
T(2^{-1}c_0)\le 2^{-1}e. 
\end{equation}
We also have
\begin{multline*}
\|T(x)\|+1=\|T(x)+e\|=\|T(x)+T(c_0)\|=\|T(x+c_0)\|
\\
=\|T(x+2^{-1}c_0)+T(2^{-1}c_0)\|\le \|T(x+2^{-1}c_0)\|+\|T(2^{-1}c_0)\|
\\
\le\|T(x)+T(2^{-1}c_0)\|+2^{-1}
\end{multline*}
for every $x\in A_+$. Thus, 
\[
    \|T(x)+2^{-1}e\|
    =\|T(x)\|+2^{-1}\le 
    \|T(x)+T(2^{-1}c_0)\|
\]
for every $x\in A_+$. 
As $T(A_+)=B_{+}$, we see by Lemma 2.6 in \cite{dongli} that 
$2^{-1}e\le T(2^{-1}c_0)$. Applying \eqref{n1}, we have $T(2^{-1}c_0)=2^{-1}c_0$.  
Suppose that \eqref{eq3.2-0} holds for $n=k$. We prove \eqref{eq3.2-0} for $n=k+1$. 
By the $\fm$-equation, we have 
\begin{multline*}
    \|T(2^{-(k+1)}c_0)\|=2^{-1}\|T(2^{-(k+1)}c_0)+T(2^{-(k+1)}c_0)\|\\
    =2^{-1}\|T(2^{-k}c_0)\|=2^{-1}\|2^{-k}e\|=2^{-(k+1)}.
\end{multline*}
Thus, 
\begin{equation}\label{n2}
T(2^{-(k+1)}c_0)\le 2^{-(k+1)}e.
\end{equation}
We also have
\begin{multline*}
    \|T(x)\|+2^{-k}=\|T(x)+2^{-k}e\|=\|T(x)+T(2^{-k}c_0)\|
    \\
    =\|T(x+2^{-k}c_0)\|=\|T(x+2^{-(k+1)}c_0)+T(2^{-(k+1)}c_0)\|
    \\
    \le \|T(x+2^{-(k+1)}c_0)\|+\|T(2^{-(k+1)}c_0)\|
    \\
    \le \|T(x)+T(2^{-(k+1)}c_0)\|+2^{-(k+1)}
\end{multline*}
for every $x\in A_+$. 
Hence, 
\[
\|T(x)+2^{-(k+1)}e\|=\|T(x)\|+2^{-(k+1)}\le \|T(x)+T(2^{-(k+1)}c_0)\|
\]
for every $x\in A_+$. Again by Lemma 2.6 in \cite{dongli}, we have 
\[
2^{-(k+1)}e\le T(2^{-(k+1)}c_0).
\]
Applying \eqref{n2}, we have 
$T(2^{-(k+1)}c_0)=2^{-(k+1)}e$. Hence, we have \eqref{eq3.2-0} by induction. 
  
        Next we prove that 
        \begin{equation}\label{eq3.2-2}
T\left(\frac{m}{2^n}c_0\right)=\frac{m}{2^n}e
        \end{equation}
        for all positive integers $m$ and $n$. 
Let $n$ be arbitrary. We prove \eqref{eq3.2-2} by induction on $m$. The case of $m=1$ has been proved. Suppose that \eqref{eq3.2-2} holds for $m=k$. We prove \eqref{eq3.2-2} for $m=k+1$. By the assumption, $T\left({\frac{k}{2^n}}c_0\right)=\frac{k}{2^n}e$, we have
\begin{multline*}
    \|T(x)+T\left(2^{-n}c_0\right)+T\left(\frac{k}{2^n}c_0\right)\|=
    \|T(x)+T\left(2^{-n}c_0\right)\|+\frac{k}{2^n}
    \\
    =
    \|T(x+2^{-n}c_0)\|+\frac{k}{2^n}=
    \|T(x+2^{-n}c_0)+T\left(\frac{k}{2^n}c_0\right)\|
    \\
    =\|T(x)+T\left(2^{-n}c_0+\frac{k}{2^n}c_0\right)\|=
    \|T(x)+T\left(\frac{k+1}{2^n}c_0\right)\|
\end{multline*}
for every $x\in A_+$. 
By Lemma 2.6 in \cite{dongli}, we have 
$T\left(2^{-n}c_0\right)+T\left(\frac{k}{2^n}c_0\right)\le T\left(\frac{k+1}{2^n}c_0\right)$ and 
$T\left(\frac{k+1}{2^n}c_0\right)\le T\left(2^{-n}c_0\right)+T\left(\frac{k}{2^n}c_0\right)$.
Thus, 
we have
\[
T\left(\frac{k+1}{2^n}c_0\right)=
T\left(2^{-n}c_0\right)+T\left(\frac{k}{2^n}c_0\right)=\frac{k+1}{2^n}e.
\]
Hence \eqref{eq3.2-2} holds for $n=k+1$. 
By induction, we have \eqref{eq3.2-2} for all positive integers $m$ and $n$.

Let $\alpha>0$ be arbitrary. We prove $T(\alpha c_0)=\alpha e$. Let $\{\alpha_n\}$ and $\{\beta_n\}$ be sequences of rational numbers of the forms $\frac{l}{2^m}$ for positive integers $l$ and $m$ such that $\alpha_n\le \alpha\le \beta_n$ and $\lim_{n\to \infty}\alpha_n=\lim_{n\to \infty}\beta_n=\alpha$. By Lemma \ref{p} and \eqref{eq3.2-2} we have
\[
\alpha_ne=T(\alpha_n c_0)\le T(\alpha c_0)\le T(\beta_n c_0)=\beta_n e.
\]
Letting $n\to \infty$, we have the conclusion. 
\end{proof}
\begin{lemma}\label{3}
For the case where $T(e)=e$,  
we have
\[
\|T(a-b)\|=\|T(a)-T(b)\|
\]
for every pair $a,b\in A_{+}$ with $b\le a$. 
\end{lemma}
\begin{proof}
    Suppose that $a,b\in A_{+}$ satisfy that $b\le a$. Then, $0\le a-b$. 
    Note that $T(a-b)$ is well defined.
    We first prove that $\|T(a-b)\|\ge \|T(a)-T(b)\|$. 
    Let $x\in A_{+}$ be arbitrary. 
    Applying the $\fm$-equation several times, we have
    \begin{multline}\label{eq3.2-1}
        \|T(a-b)\|+\|T(b)+T(x)\|=\|T(a-b)\|+\|T(b+x)\| \\
        \ge\|T(a-b)+T(b+x)\| =
        \|T(a+x)\|=\|T(a)+T(x)\|
    \end{multline}
    for every $x\in A_+$. Letting $t=\|T(b)\|+1$, we observe that $0< te-T(b)$. Recall that $0<y$ means that $0\le y$ and $y$ is invertible. 
    As $T$ is surjective, there exists $x\in A_+$ such that  
     $T(x)=te-T(b)$. Applying $T(x)=te-T(b)$ in \eqref{eq3.2-1}, we have
    \[
    \|T(a-b)\|+\|T(b)+te-T(b)\|\ge \|T(a)+te-T(b)\|,
    \]
    hence
    \[
    \|T(a-b)\|+t\ge \|T(a)-T(b)+te\|.
    \]
    By Lemma \ref{p}, we have $0\le T(a)-T(b)$. Thus
    \[
    \|T(a-b)\|+t\ge \|T(a)-T(b)\|+t,
    \]
    hence we get the desired inequality
    \[
    \|T(a-b)\|\ge \|T(a)-T(b)\|.
    \]

    We prove the reverse inequality. 
    Let $t=\|b\|+1$. Then $0< te-b$.
   Let $x=te-b$. As we consider the case where $T(e)=e$, we have $T(b+x)=T(te)=te$ by Lemma \ref{4}.
      We have by the $\fm$-equation that 
   \begin{multline}\label{eq3-1}
   \|T(a-b)\|+t=\|T(a-b)+te\|\\
   =\|T(a-b)+T(b+x)\|=\|T(a+x)\|=\|T(a)+T(x)\|,
   \end{multline}
   and 
   \[
   \|T(x)+T(b)\|=\|T(x+b)\|=\|te\|=t.
   \]
   Hence $T(x)+T(b)\le te$, so $T(x)\le te-T(b)$. Thus 
   \[
   T(a)+T(x)\le T(a)+te-T(b).
   \]
   As $0\le T(a)-T(b)$ by Lemma \ref{p}, we have
   \begin{equation}\label{eq3-2}
   \|T(a)+T(x)\|\le \|T(a)-T(b)+te\|=\|T(a)-T(b)\|+t.
   \end{equation}
   Combining \eqref{eq3-1} and \eqref{eq3-2}, we get $$\|T(a-b)\|\le \|T(a)-T(b)\|$$. 
\end{proof}

\begin{lemma}\label{5}
For every positive real number $\lambda$ and $a\in A_+$, we have
\[
T(a+\lambda c_0)=T(a)+\lambda e.
\]
\end{lemma}
\begin{proof}
    For any $x\in A_+$, we have by the $\fm$-equation and Lemma \ref{4} that
    \begin{multline}\label{eq5-1}
    \|T(a+\lambda c_0)+T(x)\|=\|T(a+\lambda c_0+x)\|
    =\|T(a+x)+T(\lambda c_0)\|
    \\
    =\|T(a+x)+\lambda e\|
    =\|T(a+x)\|+\lambda =\|T(a)+T(x)\|+\lambda 
    \end{multline}
    for every $x\in A_+$. Put $t=\|T(a)\|+1$. Then $0< te- T(a)$. As $T$ is surjective, there exists $x\in A_+$ such that $T(x)=te-T(a)$. Substituting $T(x)=te-T(a)$ for \eqref{eq5-1}, we obtain
    \[
    \|T(a+\lambda c_0)+te-T(a)\|=\|T(a)+te-T(a)\|+\lambda=t+\lambda.
    \]
    Hence we get $T(a+\lambda c_0)+te-T(a)\le (t+\lambda)e$, so 
    \begin{equation}\label{eq5-2}
    T(a+\lambda c_0)\le T(a)+\lambda e. 
    \end{equation} 

Next, substituting $T(x)=se-T(a+\lambda c_0)$, where $s=\|T(a+\lambda c_0)\|+1$ with a suitable $x\in A_+$, for \eqref{eq5-1}, we get 
    \[
    s=\|T(a)+se-T(a+\lambda c_0)\|+\lambda.
    \]
    Hence, we have $T(a)+se-T(a+\lambda c_0)\le (s-\lambda)e$, so
    \begin{equation}\label{eq5-3}
    T(a)+\lambda e\le T(a+\lambda c_0).
    \end{equation}
    By \eqref{eq5-2} and \eqref{eq5-3}, we have $T(a+\lambda c_0)=T(a)+\lambda e$.
\end{proof}
\begin{lemma}\label{++}
    We have $T(A_{++})\subset B_{++}$.
\end{lemma}
\begin{proof}
    Let $a\in A_{++}$. 
    As $a$ is invertible, the minimum spectrum of $a$ is positive. Hence there exists $t>0$ such that $tc_0\le a$. 
    Then by Lemmata \ref{p} and \ref{5} we have 
    \[
    te=T(tc_0)\le T(a).
    \]
    Hence $T(a)$ is invertible. As $a\in A_{++}$ is arbitrary, we have that $T(A_{++})\subset B_{++}$. 
\end{proof}
\begin{lemma}\label{e13,e14,21}
Suppose that $T^{-1}(B_{++})\subset A_{++}$. 
If $T(e)=e$, then $\|T(a)\|=\|a\|$ for every $a\in A_{+}$.
\end{lemma}
\begin{proof}
Let $a\in A_{+}$. 
    First, we prove that $\|T(a)\|\le \|a\|$. Put $\lambda=\|a\|$. Then, $a\le \lambda e$. By Lemma \ref{p}, we infer that $T(a)\le T(\lambda e)$. We have $T(\lambda e)=\lambda e$ by Lemma \ref{4} so that
    $T(a)\le \lambda e$. Thus, $\|T(a)\|\le \|a\|$. 
   
We prove $\|T(a)\|=\|a\|$. 
   If $a=0$, then $T(a)=0$. Hence $\|T(a)\|=\|a\|$. Suppose that $a\ne 0$. 
    Put $\lambda=\|a\|$. Since $0\le a$, we have $\lambda\in \sigma(a)$, so $0\in \sigma(\lambda e-a)$. Hence, $0\le \lambda e-a\le \lambda e$. We also have that $\lambda e-a$ is not invertible since $0\in \sigma(\lambda e-a)$. 
     We infer that $T(\lambda e-a)$ is not invertible since $T^{-1}(B_{++})\subset A_{++}$, hence $0\in \sigma(T(\lambda e-a))$, and so $\lambda\in \sigma(\lambda e-T(\lambda e-a))$.
     On the other hand, we have by the first part that 
     \[
     \|T(\lambda e-a)\|\le\|\lambda e-a\|\le \|\lambda e\|=\lambda.
     \]
     Hence, $T(\lambda e-a)\le \lambda e$. Thus
     \[
     0\le \lambda e-T(\lambda e-a)\le \lambda e.
     \]
     Hence, $\|\lambda e-T(\lambda e-a)\|\le \lambda$. 
     We have already seen that $\lambda\in \sigma(\lambda e-T(\lambda e-a))$. It follows that 
     \[
     \|\lambda e-T(\lambda e-a)\|=\lambda=\|a\|.
     \]
     By Lemmata \ref{4} and \ref{3}, we have
     \[
     \|T(a)\|=\|T(\lambda e-(\lambda e-a))\|=\|T(\lambda e)-T(\lambda e-a)\|=\|\lambda e-T(\lambda e-a)\|.
     \]
     We conclude that $\|T(a)\|=\|a\|$. 
\end{proof}
\begin{lemma}\label{p23}
Suppose further that $T$ is a bijection. 
Suppose that $a, b\in A_+$ and $T(a)+T(b)=\|T(a)+T(b)\|e$. Then $T(a+b)=\|T(a)+T(b)\|e$. 
\end{lemma}
\begin{proof}
As $0\le \|T(a+b)\|e-T(a+b)$, there exists $c\in A_+$ such that 
\begin{equation}\label{p23-1}
T(c)=\|T(a+b)\|e-T(a+b).
\end{equation}
Thus,  
$T(a+b)+T(c)=\|T(a+b)\|e$, hence $\|T(a+b)+T(c)\|=\|T(a+b)\|$. 
As $T$ satisfies the $\fm$-equation, we see that 
\begin{multline*}
\|T(a)+T(b+c)\|=\|T(a+b+c)\|
\\
=\|T(a+b)+T(c)\|
=\|T(a+b)\|=\|T(a)+T(b)\|.
\end{multline*}
Thus
\[
T(a)+T(b+c)\le \|T(a)+T(b)\|e=T(a)+T(b)
\]
by the assumption. 
Since $b\le b+c$, it follows by Lemma \ref{p} that 
\[
T(a)+T(b)\le T(a)+T(b+c).
\]
Thus, we see that 
\[
T(a)+T(b)=T(a)+T(b+c).
\]
As $T$ is an injection, we observe that $c=0$. 
By Lemma \ref{0}, we have $T(c)=0$. 
We conclude by \eqref{p23-1} that 
\[
T(a+b)=\|T(a)+T(b)\|e.
\]
\end{proof}
\begin{lemma}\label{p24}
Suppose further that $T$ is a bijection. Then $c_0$ is invertible.
\end{lemma}
\begin{proof}
   As $0\le \|T(e)\|e-T(e)$, there exist $b\in A_+$ such that $T(b)=\|T(e)\|e-T(e)$. Then, we have
   $T(e)+T(b)=\|T(e)\|e$ and $\|T(e)+T(b)\|=\|T(e)\|$. Thus, 
   \[
   T(e)+T(b)=\|T(e)+T(b)\|e.
   \]
   By Lemmata \ref{4} and \ref{p23}, we observe that
   \[
   T(e+b)=\|T(e)+T(b)\|e=\|T(e)\|e=T(\|T(e)\|c_0).
   \]
   Since $T$ is injective, we obtain
   \[
   e\le e+b=\|T(e)\|c_0.
   \]
It follows that $c_0$ is invertible. 
\end{proof}
\begin{lemma}\label{10',11'}
Suppose further that $T$ is a bijection.
  Then $T^{-1}(B_{++})=A_{++}$.
\end{lemma}
\begin{proof}
    We prove $T(A_{++})=B_{++}$. 
    Due to Lemma \ref{++} we already have $T(A_{++})\subset B_{++}$.

    To prove the reverse inclusion, we apply the invertibility of $c_0$, which is established by Lemma \ref{p24}. Let $b\in B_{++}$ be arbitrary. Since the minimum spectrum of $b$ is positive, there exists $s>0$ such that $se\le b$. Since $T(A_+)=B_+$, there exists $d\in A_+$ such that $T(d)=b-se$. Thus $b=T(d)+se$. By Lemma \ref{5}, we have $T(d+sc_0)=T(d)+se=b$. Since $d+sc_0\in A_{++}$ as $c_0$ is invertible, we conclude that $b\in T(A_{++})$. 
\end{proof}
\section{Proofs of the main Theorems and Corollary \ref{mean}}\label{s3}
We prove Theorems \ref{psdc} and \ref{ontopsdc} simultaneously, using the lemmata from the previous section.
\begin{proof}[Proof of Theorems \ref{psdc} and \ref{ontopsdc}]
If there is a Jordan $*$-isomorphism $J\colon A\to B$ with $T(\cdot)=T(e)^\frac12J(\cdot)T(e)^\frac12$, then $T$ 
satisfies the $\fm$-equation since $J$ is linear. 

We prove the converse statement. 
Let $c_0\in T^{-1}(\{e\})$. 
Then $c_0$ is an invertible element in $A$ by Lemma \ref{p24} in the case of Theorem \ref{psdc}, and by the assumption in the case of Theorem \ref{ontopsdc}. 
Put $T_0\colon A_{+}\to B_{+}$ be defined as $T_0(x)=T(c_0^\frac12xc_0^\frac12)$ for $x\in A_+$. 
As $c_0$ is invertible, we infer that $T_0\colon A_{+}\to B_{+}$ is a bijective (in the case of Theorem \ref{psdc}) (resp. surjective (in the case of Theorem \ref{ontopsdc})) $\fm$-map with $T_0(e)=e$. 
If $T$ is a surjection with $T^{-1}(B_{++})\subset A_{++}$ (in the case of Theorem \ref{ontopsdc}), then $T_0$ is also a surjection with $T_0^{-1}(B_{++})\subset A_{++}$. 
If $T$ is a bijection (in the case of Theorem \ref{psdc}), then $T_0$ is also a bijection. In this case, Lemma \ref{10',11'} ensures that $T_0^{-1}(B_{++})=A_{++}$. 
Hence, either in the case of Theorems \ref{psdc} or \ref{ontopsdc}, we have by Lemma \ref{e13,e14,21} that 
\[
\|T_0(a)+T_0(b)\|=\|T_0(a+b)\|=\|a+b\|
\]
for every pair $a, b\in A_{+}$. 
Applying Theorem 2.5 in \cite{dongli}, there exists a Jordan $*$-isomorphism $J\colon A\to B$ which extends $T_0$; $T_0=J$ on $A_+$. 
It is well known that a Jordan $*$-isomorphism preserves Jordan triple products, and we have
\[
T(a)=T_0(c_0^{-\frac12}ac_0^{-\frac12})=J(c_0^{-\frac12}ac_0^{-\frac12})=J(c_0^{-\frac12})J(a)J(c_0^{-\frac12})
\]
for every $a\in A_+$. 
In particular, we have $T(e)=J(c_0^{-\frac12})^2$. As a Jordan $*$-isomorphism preserves the positivity, we have $T(e)^\frac12=J(c_0^{-\frac12})$. 
Thus, we have the conclusion.
\end{proof}
Note that the proof of \cite[Theorem 2.5]{dongli} ultimately traces back to Kadison's theorem \cite[Theorem 2 and Corollary 5]{kadison}. Indeed, it relies on \cite[Theorem 13]{mol19}, which is a consequence of \cite[Theorem 9]{hamo}; the proof of the latter, in turn, relies heavily on \cite[Theorem 2 and Corollary 5]{kadison}.
\begin{proof}[Proof of Corollary \ref{mean}]
It is evident that (5) implines (1) and (1) implines (3).
The equivalence of (3) and (5) is due to Theorem \ref{psdc}. 
Hence (1), (3) and (5) are equivalent. 
We only need to prove that (4) implies (3). Then the rest are trivial. 

Suppose that (4) holds. We prove (3).  
We first show that $T(0)=0$. 
As $T$ is surjective, there exists $a_0\in A_+$ such that $T(a_0)=0$. Then we have
\[
0=\|T(a_0)\|=\left\|T\left(\frac{2a_0+0}{2}\right)\right\|
=\left\|\frac{T(2a_0)+T(0)}{2}\right\|\ge\left\|\frac{T(0)}{2}\right\|\ge0.
\]
It follows that $T(0)=0$. Let $c\in A_+$ be arbitrary. We have
\[
2\|T(c)\|=2\left\|T\left(\frac{2c+0}{2}\right)\right\|=2\left\|\frac{T(2c)+T(0)}{2}\right\|=\|T(2c)\|.
\]
Let $a,b\in A_+$ be arbitrary. 
Then, 
\[
\|T(a)+T(b)\|=2\left\|T\left(\frac{a+b}{2}\right)\right\|=\|T(a+b)\|.
\]
\end{proof}
\section{Proof of Corollary \ref{pc}}\label{s4}
Before proving Corollary \ref{pc}, we present Lemmata \ref{pp}, \ref{44}, \ref{33}, and \ref{55} for maps on positive definite cones. These are variants of Lemmata \ref{p}, \ref{4}, \ref{3}, and \ref{5}, where we dealt with maps on the positive semidefinite cone.

Let $S\colon A_{++}\to B_{++}$ be a surjection which satisfies the $\fm$ equation $\|S(a+b)\|=\|S(a)+S(b)\|$ for every pair $a,b\in A_{++}$. Suppose that $S(e)=e$. We have Lemmata \ref{pp} through \ref{55}.
\begin{lemma}\label{pp}
    The map $S$ preserves the order; i.e., $a\le b$ ensures that $S(a)\le S(b)$ for every pair $a,b\in A_{++}$.
\end{lemma}
\begin{proof}
A proof is similar to that of Lemma \ref{p} although we need some modification. We first observe that 
\begin{equation}\label{pp1}
S(2^{-n}e)=2^{-n}e
\end{equation}  
for every positive integer $n$. We can prove \eqref{pp1} in the same way as the proof of \eqref{eq3.2-0}, and it is omitted. 

We next prove that 
\begin{equation}\label{pp2}
S(b+2^{-n}e)=S(b)+2^{-n}e
\end{equation}
for every $b\in A_{++}$ and a positive integer $n$. 
A proof is similar to that of Lemma \ref{5}. Instead of applying Lemma \ref{4}, we use \eqref{pp1}. For every $x\in A_{++}$ and a positive integer $n$, applying the $\fm$-equation and \eqref{pp1} we have  
\begin{multline}\label{pp3}
    \|S(b+2^{-n}e)+S(x)\|
    =\|S(b+x)+S(2^{-n}e)\|
    \\
    =\|S(b+x)+2^{-n}e\|
    =\|S(b+x)\|+2^{-n} 
    \\
    =\|S(b)+S(x)\|+2^{-n}.
\end{multline}
Put $t=\|S(b)\|+1$. Then $0<te-S(b)$. As $S$ is surjective, there exists $x\in A_{++}$ with $S(x)=te-S(b)$. 
Substituting $S(x)=te-S(b)$ for \eqref{pp3}, we get 
\[
\|S(b+2^{-n}e)+te-S(b)\|=t+2^{-n}.
\]
Thus we have
\begin{equation}\label{pp4}
S(b+2^{-n}e)\le S(b)+2^{-n}e.
\end{equation}
Substituting $S(x)=se-S(b+2^{-n}e)$, where $s=\|S(b+2^{-n}e)\|+1$ for \eqref{pp3}, we obtain
\begin{equation}\label{pp5}
S(b)+2^{-n}e\le S(b+2^{-n}e)
\end{equation}
in a similar way as the proof of \eqref{eq5-3}. By \eqref{pp4} and \eqref{pp5} we have \eqref{pp2}. 

Let $a,b\in A_{++}$ be $a\le b$. We have $0<b+2^{-n}e-a$ for every positive integer $n$. In a similar way to the proof of Lemma \ref{p}, we have by \eqref{pp2} that
\[
S(a)\le S(b+2^{-n}e)=S(b)+2^{-n}e.
\]
Letting $n\to \infty$ we have the conclusion. 
\end{proof}
\begin{lemma}\label{44}
    For any positive $\alpha$, we have $S(\alpha e)=\alpha e$.
\end{lemma}
\begin{lemma}\label{33}
    For every pair $a,b\in A_{++}$ with $0<a-b$, we have
    \[
    \|S(a-b)\|=\|S(a)-S(b)\|.
    \]
\end{lemma}
\begin{lemma}\label{55}
For every positive real number $\lambda$ and $a\in A_{++}$, we have
\[
S(a+\lambda e)=S(a)+\lambda e.
\]
\end{lemma}
Proofs of Lemmata \ref{44}, \ref{33}, and \ref{55} are similar to those of Lemmata \ref{4}, \ref{3}, \ref{5} and are omitted.
\begin{proof}[Proof of Corollary \ref{pc}]
If there is a Jordan $*$-isomorphism $J\colon A\to B$ with $T(\cdot)=T(e)^\frac12J(\cdot)T(e)^\frac12$, then $T$ is additive since $J$ is linear. Hence $T$  
satisfies $\fm$-equation. 

We prove the converse statement.
Put $c_0=T^{-1}(e)$. Define $T_0\colon A_{++}\to B_{++}$ by $T_0(x)=T(c_0^\frac12xc_0^\frac12)$ for $x\in A_{++}$. 
As $c_0$ is invertible, it is easy to see that $T_0$ is a bijection and that $T_0(e)=e$. It is also plain that $T_0$ satisfies the $\fm$-equation. 
Extend $T_0$ to $\tilde{T_0}\colon A_+\to B_+$ by 
\[
\tilde{T_0}(a)=\lim_{n\to \infty}T_0\left(a+2^{-n}e\right)
\]
Then we have
\begin{itemize}
    \item[(1)] $\tilde{T_0}\colon A_+\to B_+$ is well defined,
    \item[(2)] $\tilde{T_0}=T_0$ on $A_{++}$,
    \item[(3)] $\tilde{T_0}$ is an injection,
    \item[(4)] $\tilde{T_0}$ is a surjection,
    \item[(5)] $\|\tilde{T_0}(a+b)\|=\|\tilde{T_0}(a)+\tilde{T_0}(b)\|$ holds for every pair $a,b\in A_+$,
    \item[(6)] $\tilde{T_0}(e)=e$.
\end{itemize}

Proof of (1). We prove that the sequence $\{T_0\left(a+2^{-n}e\right)\}_{n=1}^\infty$ is a Cauchy sequence. Indeed, for positive integers $n<m$, we have by Lemmata \ref{44} and \ref{33} that 
\begin{multline*}
    \|T_0\left(a+2^{-n}e\right)-T_0\left(a+2^{-m}e\right)\|
    \\
    = \|T_0\left(2^{-n}e-2^{-m}e\right)\|
    =
    \|\left(2^{-n}-2^{-m}\right)e\| 
    =2^{-n}-2^{-m}.
\end{multline*}
Thus 
$\{T_0\left(a+2^{-n}e\right)\}_{n=1}^\infty$ 
is a Cauchy sequence. 
Since $B_+$ is complete with $\|\cdot\|$, $\lim_{n\to \infty}T_0\left(a+2^{-n}e\right)$ exists. Hence, the map $\tilde{T_0}\colon A_+\to B_+$ is well defined. 

Proof of (2). Let $a\in A_{++}$ be arbitrary. By Lemmata \ref{44} and \ref{33}, we have  
\[
\|T_0\left(a+2^{-n}e\right)-T_0(a)\|
=
\|T_0\left(2^{-n}e\right)\|=2^{-n}.
\]
Hence, 
\[
\tilde{T_0}(a)=\lim_{n\to \infty}T_0\left(a+2^{-n}e\right)=T_0(a).
\]

Proof of (3) Let $a,b\in A_+$ be $a\ne b$. We prove $\tilde{T_0}(a)\ne \tilde{T_0}(b)$. For every positive integer $n$, $a+2^{-n}e\ne b+2^{-n}e$. 
By Lemma \ref{55} we have
\[
T_0\left(a+2^{-n}e\right)=T_0\left(a+2^{-(n+1)}e+2^{-(n+1)}e\right)=T_0\left(a+2^{-(n+1)}e\right)+2^{-(n+1)}e
\]
and 
\[
T_0\left(b+2^{-n}e\right)=T_0\left(b+2^{-(n+1)}e+2^{-(n+1)}e\right)=T_0\left(b+2^{-(n+1)}e\right)+2^{-(n+1)}e.
\]
Hence,
\[
T_0\left(a+2^{-n}e\right)-T_0\left(b+2^{-n}e\right)
=
T_0\left(a+2^{-(n+1)}e\right)-T_0\left(b+2^{-(n+1)}e\right)
\]
for every $n$. Thus,
\[
T_0\left(a+2^{-1}e\right)-T_0\left(b+2^{-1}e\right)
=
T_0\left(a+2^{-n}e\right)-T_0\left(b+2^{-n}e\right)
\]
for every positive integer $n$. It follows that 
\begin{multline*}
T_0\left(a+2^{-1}e\right)-T_0\left(b+2^{-1}e\right)
\\
=
\lim_{n\to \infty}\left(T_0\left(a+2^{-n}e\right)-T\left(b+2^{-n}e\right)\right)=\tilde{T_0}(a)-\tilde{T_0}(b).
\end{multline*}
We have $\tilde{T_0}(a)\ne \tilde{T_0}(b)$ as $T_0$ is injective and 
$T_0\left(a+2^{-1}e\right)\ne T_0\left(b+2^{-1}e\right)$.

Proof of (4). Let $b\in B_{+}$ be arbitrary. 
By the bijectivity of $T_0$, there is a unique $a_n\in A_{++}$ such that $T_0(a_n)=b+2^{-n}e$ for each positive integer $n$. By Lemma \ref{55}, we have 
\begin{multline*}
    T_0\left(a_{n+1}+2^{-(n+1)}e\right)=T_0(a_{n+1})+2^{-(n+1)}e
    \\
    =b+2^{-(n+1)}e+2^{-(n+1)}e=
    b+2^{-n}e=T_0(a_n).
\end{multline*}
By the injectivity of $T_0$, we get
\[
a_{n+1}+2^{-(n+1)}e=a_n
\]
for every positive integer $n$. It follows that 
\[
a_n+\left(2^{-1}-2^{-n}\right)e=a_1,
\]
hence
\[
0\le \lim_{n\to \infty}a_n=a_1-2^{-1}e,
\]
hence, $a_1-2^{-1}e\in A_+$. 
It follows that
\begin{multline*}
\tilde{T_0}\left(a_1-2^{-1} e\right)=\lim_{n\to \infty}T_0\left(a_1-2^{-1}e+2^{-n}e\right)
\\
=\lim_{n\to \infty}T_0(a_n)=\lim_{n\to \infty}\left(b+2^{-n}e\right)=b.
\end{multline*}
As $b$ is arbitrary, we have that $\tilde{T_0}$ is surjective.

Proof of (5). Let $a,b\in A_+$ be arbitrary. 
For every positive integer $n$, we have 
$a+2^{-(n+1)}e, b+2^{-(n+1)}e, a+b+2^{-n}e\in A_{++}$. Thus we have
\[
\|T_0\left(a+2^{-(n+1)}e\right)+T_0\left(b+2^{-(n+1)}e\right)\|=\|T_0\left(a+b+2^{-n}e\right)\|.
\]
Letting $n\to \infty$, we get
\[
\|\tilde{T_0}(a)+\tilde{T_0}(b)\|=\|\tilde{T_0}(a+b)\|.
\]

(6) is easy. By (2) we have $\tilde{T_0}(e)=T_0(e)=e$. 

Applying Theorem \ref{psdc} for $\tilde{T_0}$, there is a Jordan $*$-isomorphism $J\colon A\to B$ such that $\tilde{T_0}=J$ on $A_+$. Hence, $T_0=J$ on $A_{++}$.
As in the same way as the proof of Theorems \ref{psdc} and \ref{ontopsdc},  
we conclude that $T(a)=T(e)^\frac12 J(a)T(e)^\frac12$ for every $a\in A_{++}$.
\end{proof}
\section{Final comments}\label{frc}
For a surjection between the positive definite cones of unital $C*$-algebras, a similar result as 
Corollary \ref{mean} is possible. In fact, we have
\begin{corollary}\label{final}
Suppose that $T\colon A_{++}\to B_{++}$ is a bijection. 
Then the following are equivalent.
\begin{itemize}
    \item[(1)] $T$ satisfies the Cauchy equation, i.e., $T(a+b)=T(a)+T(b)$ for every pair $a,b\in A_{++}$,
    \item[(2)] $T$ satisfies the Jensen equation, i.e., 
    $T\left(\frac{a+b}{2}\right)=\frac{T(a)+T(b)}{2}$ for every pair $a,b\in A_{++}$,
    \item[(3)]$T$ satisfies the $\fm$-equation, i.e., $\|T(a+b)\|=\|T(a)+T(b)\|$ for every pair $a,b\in A_{++}$, 
    \item[(5)] there is a Jordan $*$-isomorphism $J\colon A\to B$ such that 
    $T(a)=T(e)^\frac12 J(a)T(e)^\frac12$ 
    for every $a\in A_{++}$.
\end{itemize}
\end{corollary}
The equivalence of (1), (3) and (5) is followed by Corollary \ref{pc}. Moln\'ar \cite[Proposition 1]{mol17} proved that (2) and (5) are equivalent to each other. On the other hand, we do not know if 
\begin{itemize}
    \item[(4)] $T$ satisfies the equation $\left\|T\left(\frac{a+b}{2}\right)\right\|=\left\|\frac{T(a)+T(b)}{2}\right\|$
for every pair $a,b\in A_{++}$
\end{itemize}
is equivalent to (5) or not.
\subsection*{Declaration}
The authors used ChatGPT (OpenAI) to assist with English-language editing. After using this tool, the authors reviewed and edited the content as needed and take full responsibility for the 
content of the publication.
\subsection*{Data Availability Statement} Data sharing is not applicable to this article as
no datasets were generated or analyzed during the current study.
\subsection*{Conflict of interest} 
On behalf of all authors, the corresponding author states that there is no conflict of interest.

\end{document}